\documentclass[a4paper
,12pt
,leqno
,twoside
]{article}

\usepackage{amsmath, amsthm, amssymb,bm}
\numberwithin{equation}{section}

\newcommand{\bC}{\mathbb{C}}

\newcommand{\bE}{\mathbb{E}}

\newcommand{\mf}[1]{\mathfrak{#1}}
\newcommand{\mr}[1]{\mathrm{#1}}
\newcommand{\mcal}[1]{\mathcal{#1}}

\def\tr{\mathrm{tr} \,}

\def\({ \left( }
\def\){ \right)}

\def\trans#1{\mathord{\mathopen{{\vphantom{#1}}^t}#1}}

\theoremstyle{plain}
\newtheorem{thm}{Theorem}[section]
\newtheorem{prop}[thm]{Proposition}
\newtheorem{lem}[thm]{Lemma}
\newtheorem{cor}[thm]{Corollary}
\theoremstyle{definition}

\newtheorem{example}{Example}
\newtheorem{remark}{Remark}
\theoremstyle{conjecture}

\theoremstyle{problem}

\title{\bfseries Moments of a single entry of circular orthogonal ensembles and Weingarten calculus
}
\author{\textsc{Sho Matsumoto} 
\thanks{The author was partially supported by Grant-in-Aid for Young Scientists (B) No. 22740060.
} \\
Graduate School of Mathematics, Nagoya University, \\
 Nagoya, 464-8602, Japan. \\
\qquad E-mail: sho-matsumoto@math.nagoya-u.ac.jp
}

\date{\empty}

\begin{document}

\maketitle

\begin{abstract}
Consider a symmetric unitary random matrix $V=(v_{ij})_{1 \le i,j \le N}$
from a circular orthogonal ensemble.
In this paper, we study moments of a single entry $v_{ij}$.
For a diagonal entry $v_{ii}$ we give the explicit values of the moments,
and for an off-diagonal entry $v_{ij}$ 
we give leading and subleading terms in the asymptotic expansion with respect to 
a large matrix size $N$.
Our technique is to apply the Weingarten calculus for 
a Haar-distributed unitary matrix.

\end{abstract}

\section{Introduction and results}

In random matrix theory, given a random matrix $X=(x_{ij})_{1 \le i,j \le N}$,
the study of distributions for its eigenvalues is the main theme. 
Another typical theme is to study joint moments for matrix entries
$\bE [x_{i_1 j_1} x_{i_2 j_2} \cdots x_{i_n j_n}]$ or
\begin{equation} \label{eq:jointmoments}
\bE [x_{i_1 j_1} x_{i_2 j_2} \cdots x_{i_n j_n}
 \overline{x_{i_1' j_1'} x_{i_2' j_2'} \cdots x_{i_m' j_m'}}].
\end{equation}
When $x_{ij}$ are Gaussian random variables, 
the technique of the calculation for the joint moments is called
the Wick calculus.
For a Wishart matrix and inverted Wishart matrix, such moments are calculated in
\cite{GLM1,GLM2, KN, Mat_Wishart}.
For a Haar-distributed matrix from  classical Lie groups,
the corresponding technique is called {\it Weingarten calculus}
and developped in \cite{C,CM,CS,Mat_JMortho,MN,Novak,W}.
In those calculuations, we employ deep combinatorics of permuatations and matchings
and representation theory of symmetric groups and hyperoctahedral groups.

There are three much-studied circular ensembles:
circular orthogonal ensembles (COEs), circular unitary ensembles (CUEs), and 
circular symplectic ensembles (CSEs).
The density function for their eigenvalues $\Lambda_1,\Lambda_2,\dots,\Lambda_N$
is proportional to $\prod_{1 \le i<j \le N} |\Lambda_i-\Lambda_j|^\beta$
with $\beta=1$ (COE), $\beta=2$ (CUE), and $\beta=4$ (CSE).
The CUE is nothing but the unitary group equipped with
its Haar probability measure. 
See \cite{Mehta} for details.

In this paper, we consider the COE,
which is the probability space of symmetric unitary matrices
by the property of being invariant under 
automorphisms
$$
V \to \trans{U_0} V U_0,
$$
where $U_0$ is a unitary matrix.
It is also a realization of the symmetric space $U(N)/O(N)$,
where $U(N)$ and $O(N)$ is the unitary and real orthogonal group of degree $N$,
respectively.
If $U$ is a Haar-distributed unitary matrix 
from the unitary group $U(N)$, i.e., if $U$ is a CUE matrix,
 a random matrix $V$ from the COE can be given by
$V= \trans{U} U$. 
In other words,
if $V=(v_{ij})$ and $U=(u_{ij})$
then
\begin{equation} \label{entry:COE-CUE}
v_{ij}= \sum_{k=1}^N u_{ki} u_{kj}.
\end{equation}
Therefore a joint moment of the form
\eqref{eq:jointmoments} with $X=V$
can be written as a sum of joint moments of the forms \eqref{eq:jointmoments} with $X=U$.
Since the joint moments for a CUE matrix $U$ can be computed by
the Weingarten calculus \cite{C,CM,MN,W},
we can compute the joint moments for the COE matrix $V$, in principle.
However, in general, their computations seem quite complicated.

In the light of that situation, 
in this paper we focus on only moments $\bE[|v_{ij}|^{2n}]$ for a single matrix entry $v_{ij}$.
Diagonal entries $v_{ii}$ and
off-diagonal entries $v_{ij}$ have different distributions.
For a diagonal entry,
we give the explicit expression for $\bE[|v_{ii}|^{2n}]$ as follows.

\begin{thm} \label{thm:COE-diagonal-moment}
Let $V=(v_{ij})_{1 \le i,j \le N}$ be an $N \times N$ COE matrix.
For positive integers $i,N,n$ with $1 \le i \le N$,
$$
\bE[|v_{ii}|^{2n}] = 
\frac{2^n n!}{(N+1)(N+3) \cdots (N+2n-1)}.
$$
\end{thm}

Unfortunately, we could not obtain a similar closed expression
for an off-diagonal entry $v_{ij}$ $(i \not= j)$.
However,
we give the leading and sub-leading terms in
the asymptotic expansion of $\bE[|v_{ij}|^{2n}]$ as $N \to \infty$.

\begin{thm} \label{thm:COE-offdiagonal-moment}
Fix positive integers $i,j,n$ with $i \not= j$. 
Let $V^{(N)}=(v_{ij}^{(N)})$, $N \ge 1$,  be a sequence of $N \times N$ COE matrices.
As the matrix size $N$ goes to the infinity,
$$
\bE[|v_{ij}^{(N)}|^{2n}]= n! \(N^{-n} - \frac{n(n+1)}{2}  N^{-n-1}\)
+O(N^{-n-2}). 
$$
\end{thm}

Note that Theorem \ref{thm:COE-diagonal-moment} implies 
$$
\bE[|v_{ii}^{(N)}|^{2n}] = 2^n n! (N^{-n} -n^2 N^{-n-1}) +O(N^{-n-2}). 
$$

The following result is obtained in \cite[Corollary 1.1]{Jiang}
in an analytic  approach.
We have its new algebraic proof via CUEs
from Theorem \ref{thm:COE-diagonal-moment} and \ref{thm:COE-offdiagonal-moment}.

\begin{cor} \label{cor:limitGaussian}
Fix positive integers $i,j,n$ with $i \not= j$. 
Let $V^{(N)}=(v_{ij}^{(N)})$, $N \ge 1$,  be a sequence of $N \times N$ COE matrices.
Then, as $N \to \infty$, both $\sqrt{N/2} v_{ii}^{(N)}$ and $\sqrt{N} v_{ij}^{(N)}$ converge to 
a standard complex Gaussian random variable.
\end{cor}

This paper is organized as follows.
In Section 2, we review the Weingarten calculus for the unitary group. 
In Section 3, the proofs for main theorems are given.

As a closing statement, we give  facts for the moment $\bE[|x_{ij}|^{2n}]$,
where $x_{ij}$ is a matrix entry of a Haar-distributed unitary or orthogonal matrix.
Compare with our theorems.
Let $1 \le i,j \le N$.
For a Haar-distributed unitary matrix $U=(u_{ij})$ of size $N$,
$$
\bE[|u_{ij}|^{2n}]= \frac{n!}{N(N+1)\cdots (N+n-1)};
$$ 
For a Haar-distributed orthogonal matrix $O=(o_{ij})$ of size $N$,
$$
\bE[o_{ij}^{2n}]= \frac{(2n-1)!!}{N(N+2) \cdots (N+2n-2)}.
$$
They are obtained in \cite{Novak,CM}.

\begin{remark}
After the submission of the present article,
the author obtained stronger results in 
\cite{Mat_COE_general}, which includes the theorems of this article.
\end{remark}

\section{Weingarten calculus for unitary groups}

\subsection{Weingarten functions}

A {\it partition} $\lambda=(\lambda_1,\lambda_2,\dots)$
of a positive integer $n$ is 
a weakly-decreasing sequence of nonnegative integers
satisfying $n=|\lambda|:=\sum_{i \ge 1} \lambda_i$.
Then we write $\lambda \vdash n$.
The number of nonzero $\lambda_i$ is called the length of $\lambda$
and written as $\ell(\lambda)$: $\ell(\lambda)=|\{i \ge 1 \ | \ \lambda_i >0\}|$.
Put $\lambda != \lambda_1 ! \lambda_2 ! \cdots$.

The {\it Weingarten function} for the unitary group $U(N)$
is a class function on the symmetric group $S_n$ and given by
$$
\mr{Wg}^{U(N)}_n(\sigma) = \frac{1}{n!}
 \sum_{\begin{subarray}{c} \lambda \vdash n \\ \ell(\lambda) \le N \end{subarray}} 
\frac{f^\lambda}{\prod_{(i,j)\in\lambda} (N+j-i)} \chi^\lambda(\sigma)
\qquad (\sigma \in S_n).
$$
Here $\chi^\lambda$ is the irreducible character of $S_n$ associated with $\lambda$
and $f^\lambda$ is its degree (i.e., $f^\lambda$ is the value of $\chi^\lambda$
at the identity permutation $\mr{id}_n$).
The product $\prod_{(i,j) \in \lambda}$ stands for 
$\prod_{i=1}^{\ell(\lambda)} \prod_{j=1}^{\lambda_i}$.
We note that $\mr{Wg}_n^{U(N)}(\sigma^{-1})=\mr{Wg}_n^{U(N)}(\sigma)$
and $\mr{Wg}_n^{U(N)}(\sigma \tau)= \mr{Wg}_n^{U(N)}(\tau\sigma)$
for all $\sigma,\tau \in S_n$.

\begin{lem}[\cite{C}] \label{lem:asymWg}
Fix a positive integer $n$ and permutation $\sigma \in S_n$.
As $N \to \infty$,
the Weingarten function $\mr{Wg}^{U(N)}_n(\sigma)$ has the following asymptotic properties. 
\begin{enumerate}
\item If $\sigma$ is the identity $\mr{id}_n$,
then $\mr{Wg}^{U(N)}_n(\mr{id}_n)= N^{-n}+ O(N^{-n-2})$;
\item If $\sigma$ is a transposition, then
$\mr{Wg}^{U(N)}_n(\sigma)= -N^{-n-1}+ O(N^{-n-3})$;
\item Otherwise, $\mr{Wg}^{U(N)}_n(\sigma)= O(N^{-n-2})$.
\end{enumerate}
\end{lem}

For  a finite set $I$ of positive integers,
denote by $S_I$ the symmetric group acting on $I$.
The {\it Young subgroup} $S_\lambda$ associated with $\lambda \vdash n$
is the subgroup of $S_n$ defined by
$$
S_\lambda= S_{\{1,2,\dots,\lambda_1\}} \times S_{\{\lambda_1+1, \lambda_1+2,\dots,
\lambda_1+\lambda_2\}} \times \cdots.
$$
In particular, $S_{(n)}=S_n$ and $S_{(1^n)}=\{ \mr{id}_n \}$.
The cardinality of $S_\lambda$ is $\lambda !$.

\begin{lem} \label{lem:sumYoungsubgroup}
Let $\mu \vdash n$ and let $S_\mu$ be its corresponding Young subgroup.
Then we have 
$$
\sum_{\sigma \in S_\mu} \mr{Wg}^{U(N)}_n(\sigma)= \frac{\mu!}{n!}
\sum_{\begin{subarray}{c} \lambda \vdash n \\ \ell(\lambda) \le N \end{subarray}} 
\frac{f^\lambda K_{\lambda \mu}}
{\prod_{(i,j) \in \lambda} (N+j-i)}.
$$
Here $K_{\lambda \mu}$ is the Kostka number (i.e. the number of semi-standard tableaux of
shape $\lambda$ and weight $\mu$; see \cite[Chapter I]{Mac}).
\end{lem}

\begin{proof}
From the definition of the Weingarten function
we have
$$
\sum_{\sigma \in S_\mu} \mr{Wg}^{U(N)}_n(\sigma)=
\frac{|S_\mu|}{n!}
 \sum_{\begin{subarray}{c} \lambda \vdash n \\ \ell(\lambda) \le N \end{subarray}} 
\frac{f^\lambda}{\prod_{(i,j)\in\lambda} (N+j-i)} 
\langle \mr{res}_{S_n}^{S_\mu} \chi^\lambda, \xi_{\mu} \rangle_{S_\mu},
$$
where $\xi_\mu$ is the trivial character of $S_\mu$ and
$\langle \cdot, \cdot \rangle_{S_\mu}$ is the scalar product
on $\bC[S_\mu]$.
It follows from the Frobenius reciprocity that 
$\langle \mr{res}_{S_n}^{S_\mu} \chi^\lambda, \xi_{\mu} \rangle_{S_\mu}=
\langle \chi^\lambda, \mr{ind}_{S_\mu}^{S_n} \xi_{\mu} \rangle_{S_n}$,
which coincides with the scalar product $\langle s_\lambda, h_\mu \rangle$ 
in the algebra of symmetric functions (see \cite[Chapter I.7]{Mac}),
where $s_\lambda$ and $h_\lambda$ are the 
Schur and complete symmetric function, respectively.
Hence, since
$\langle s_\lambda, h_\mu \rangle = K_{\lambda \mu}$ (see \cite[Chapter I (5.14)]{Mac}),
we have 
$\langle \mr{res}_{S_n}^{S_\mu} \chi^\lambda, \xi_{\mu} \rangle_{S_\mu}=K_{\lambda \mu}$.
\end{proof}

\subsection{Some identities for unitary matrix integrals}

For a permutation $\sigma \in S_n$ and a sequence $\bm{i}=(i_1,\dots,i_n) \in [N]^n$,
we put
$$
\bm{i}^{\sigma}= (i_{\sigma(1)},\dots,i_{\sigma(n)}). 
$$
This gives the action of $S_n$ on $[N]^n$ from the right.

For two sequences $\bm{i}=(i_1,\dots,i_n)$
and $\bm{j}=(j_1,\dots,j_n)$ in $[N]^n$,
we write $\bm{i} \sim \bm{j}$ if
$\bm{j}$ is a permutation of $\bm{i}$,
so that there exists a permutation $\sigma \in S_n$ 
satisfying $\bm{j}=\bm{i}^\sigma$.
For example, $(2,3,2,1,3) \sim (3,3,2,2,1)$.

Given $U =(u_{ij})\in U(N)$ and 
four sequences 
$\bm{i}=(i_1,i_2,\dots,i_n)$, $\bm{j}=(j_1,\dots,j_n)$, 
$\bm{i}'=(i_1',\dots,i_m')$,
$\bm{j}'=(j_1,\dots,j_m')$ of indices in $[N]$,
we put
$$
U(\bm{i},\bm{j}|\bm{i}',\bm{j}')=
u_{i_1 j_1} u_{i_2 j_2} \cdots u_{i_n j_n} 
\overline{u_{i_1' j_1'} u_{i_2' j_2'} \cdots u_{i_m' j_m'}}.
$$

\begin{lem}[\cite{C}] \label{lem:WgFormula}
Let $U$ be an $N \times N$ Haar-distributed unitary matrix and let
$\bm{i}=(i_1,i_2,\dots,i_n)$, $\bm{j}=(j_1,\dots,j_n)$, 
$\bm{i}'=(i_1',\dots,i_m')$,
$\bm{j}'=(j_1,\dots,j_m')$ be four sequences of indices in $[N]$.
Then 
$\bE[U(\bm{i},\bm{j}|\bm{i}',\bm{j}')]$ vanishes unless
$n=m$, $\bm{i} \sim \bm{i}'$, and $\bm{j} \sim \bm{j}'$.
In this case, 
$$
\bE[U(\bm{i},\bm{j}|\bm{i}',\bm{j}')]= 
\sum_{\begin{subarray}{c} \sigma \in S_n \\ \bm{i}^\sigma=\bm{i}' \end{subarray}} 
\sum_{\begin{subarray}{c} \tau \in S_n \\ \bm{j}^\tau=\bm{j}' \end{subarray}} \mr{Wg}^{U(N)}_n(\sigma \tau^{-1}).
$$
\end{lem}

A sequence $\bm{i} \in [N]^n$ can be written as a rearrangement of a sequence of the form
$$
(\underbrace{l_1,l_1,\dots,l_1}_{\lambda_1}, \underbrace{l_2,\dots, l_2}_{\lambda_2}, \dots,
\underbrace{l_k,\dots,l_k}_{\lambda_k}), 
$$
where $l_1,\dots,l_k \in [N]$ are all distinct and $\lambda=(\lambda_1,\lambda_2,\dots,\lambda_k)$
is a partition of $n$.
Then we call $\lambda$ the {\it type} of 
$\bm{i}$.
For example, both
$(7,7,7,2,2,2,3)$ and $(2,5,4,2,5,5,2)$ are of type $(3,3,1) \vdash 7$.
Two sequences $\bm{i}$ and $\bm{j}$ have the same type
 if $\bm{i} \sim \bm{j}$.

\begin{prop} \label{CUE:row1}
Let $U=(u_{ij})_{1 \le i,j \le N}$ be a Haar-distributed unitary matrix from $U(N)$ and 
let $\bm{i} \in [N]^n$ be a sequence of type $\mu$. 
Put $(1^n):=(1,1,\dots,1)$ with $n$ times.
Then
$$
\bE[|u_{1 i_1} u_{1 i_2} \cdots u_{1 i_n}|^2]=
\bE[U((1^n),\bm{i}|(1^n),\bm{i})]=
\frac{\mu!}{N(N+1) \cdots (N+n-1)}.
$$
\end{prop}

\begin{proof}
Let $\bm{i}_\mu$ be the sequence in $[N]^n$ given by
$$
\bm{i}_\mu=
(\underbrace{1,1,\dots,1}_{\mu_1}, \underbrace{2,2,\dots, 2}_{\mu_2}, \dots,
\underbrace{l,l,\dots,l}_{\mu_l}),  \qquad l=\ell(\mu).
$$
For each $\sigma \in S_n$,
$(\bm{i}_\mu)^{\sigma}=\bm{i}_\mu$ if and only if
$\sigma \in S_\mu$.
Therefore it follows from Lemma \ref{lem:WgFormula} that 
\begin{align*}
& \bE[U((1^n),\bm{i}|(1^n),\bm{i})] =
\bE[U((1^n),\bm{i}_\mu|(1^n),\bm{i}_\mu)] \\
=&\sum_{\sigma \in S_n}
\sum_{\tau \in S_\mu} \mr{Wg}^{U(N)}_n (\sigma \tau^{-1})=
|S_\mu| \sum_{\sigma \in S_n} \mr{Wg}^{U(N)}_n (\sigma).
\end{align*}
Since $K_{\lambda, (n)}=\delta_{\lambda,(n)}$ and $f^{(n)}=1$, 
Lemma \ref{lem:sumYoungsubgroup} at $\mu=(n)$ shows
$$
\sum_{\sigma \in S_n} \mr{Wg}^{U(N)}_n (\sigma)
= \frac{1}{N(N+1)(N+2) \cdots (N+n-1)}
$$
from which the result follows.
\end{proof}

We are not going to apply the following proposition.
However, it is important because it gives a joint moment for diagonal entries of
a CUE matrix $U$.

\begin{prop} \label{prop:CUErow2}
Let $U=(u_{ij})_{1 \le i,j \le N}$ be a Haar-distributed unitary matrix
from $U(N)$
 and 
let $\bm{i}=(i_1,\dots,i_n) \in [N]^n$ be a sequence of type $\mu$. Then
$$
\bE[|u_{i_1 i_1} u_{i_2 i_2} \cdots u_{i_n i_n}|^2]=
\bE[U(\bm{i},\bm{i}|\bm{i},\bm{i})]=
\frac{(\mu !)^2}{n!} \sum_{\begin{subarray}{c} \lambda \vdash n \\ \ell(\lambda) \le N \end{subarray}} 
\frac{f^\lambda K_{\lambda \mu}}
{\prod_{(i,j) \in \lambda} (N+j-i)}.
$$
\end{prop}

\begin{proof}
The discussion similar to the previous proposition gives
$$
\bE[U(\bm{i},\bm{i}|\bm{i},\bm{i})] =
\bE[U(\bm{i}_\mu,\bm{i}_\mu|\bm{i}_\mu,\bm{i}_\mu)]
=\sum_{\sigma \in S_\mu} \sum_{\tau \in S_\mu} \mr{Wg}^{U(N)}_n(\sigma \tau^{-1})
= |S_\mu|\sum_{\sigma \in S_\mu}\mr{Wg}^{U(N)}_n(\sigma),
$$
from which the statement follows by Lemma \ref{lem:sumYoungsubgroup}.
\end{proof}

\section{Proofs of main theorems}

\subsection{General properties}

We use $X \stackrel{\mathrm{d}}{=} Y$ to denote that
random variables (or random matrices) $X$ and $Y$ have the same distribution.
Let $V=(v_{ij})_{1 \le i,j \le N}$ be a COE matrix.
Let $M(\pi)=(\delta_{i,\pi(j)})_{1 \le i,j \le N}$ be the permutation matrix
associated to $\pi \in S_{N}$.
Since $V \stackrel{\mathrm{d}}{=} \trans{M(\pi)} V M(\pi)$
by the invariant property for the COE, we have
$V \stackrel{\mathrm{d}}{=} (v_{\pi(i),\pi(j)})_{1 \le i,j \le N}$.
In particular, if $1 \le i \not=j \le N$, then
$v_{ii} \stackrel{\mathrm{d}}{=} v_{11}$ and $v_{ij} \stackrel{\mathrm{d}}{=} v_{12}$.

Using a Haar-distributed unitary matrix $U=(u_{ij})_{1 \le i,j \le N}$,
we can write $V=\trans{U} U$.
Let  $\bm{j}=(j_1,j_2,\dots,j_{2n}) \in [N]^{2n}$ and $\bm{j}'=(j_1',j_2',\dots,j_{2m}') \in 
[N]^{2m}$.
Since $v_{ij}=\sum_k u_{ki} u_{kj}$,
we have
$$
\bE[v_{j_1, j_2} v_{j_3, j_4} \cdots v_{j_{2n-1}, j_{2n}} 
\overline{v_{j_1', j_2'} v_{j_3', j_4'} \cdots v_{j_{2m-1}', j_{2m}'} }]
=\sum_{\bm{k} \in [N]^n} \sum_{\bm{k}' \in [N]^m}
\bE[U(\tilde{\bm{k}}, \bm{j}| \tilde{\bm{k}}',\bm{j}')],
$$
where 
\begin{align*}
\bm{k}=& (k_1,k_2,\dots,k_n), & \bm{k}'=& (k_1',k_2',\dots,k_m'), \\
\tilde{\bm{k}}=&(k_1,k_1,k_2,k_2,\dots,k_n,k_n), &  
\tilde{\bm{k}}'=&(k_1',k_1',k_2',k_2',\dots,k_m',k_m') .
\end{align*}
Hence, by Lemma \ref{lem:WgFormula} we have

\begin{lem} \label{lem:COEcondition}
$\bE[v_{j_1, j_2} v_{j_3, j_4} \cdots v_{j_{2n-1}, j_{2n}} 
\overline{v_{j_1', j_2'} v_{j_3', j_4'} \cdots v_{j_{2m-1}', j_{2m}'} }]$ vanishes
unless $n=m$ and $\bm{j} \sim \bm{j}'$.
\end{lem}

We now suppose $\bm{j} \sim \bm{j}'$.
By Lemma \ref{lem:WgFormula} again, 
$\bE[U(\tilde{\bm{k}}, \bm{j}; \tilde{\bm{k}}',\bm{j}')]$
vanishes unless $\tilde{\bm{k}} \sim \tilde{\bm{k}'}$,
so that unless $\bm{k} \sim \bm{k}'$. Therefore
\begin{align}
&\bE[v_{j_1, j_2} v_{j_3, j_4} \cdots v_{j_{2n-1}, j_{2n}} 
\overline{v_{j_1', j_2'} v_{j_3', j_4'} \cdots v_{j_{2n-1}', j_{2n}'} }] 
\label{eq:COEexpand1} \\
=&\sum_{\bm{k} \in [N]^n} 
\sum_{
\begin{subarray}{c} \bm{k}' \in [N]^n \\ \bm{k}' \sim \bm{k} \end{subarray}}
\bE[U(\tilde{\bm{k}}, \bm{j}| \tilde{\bm{k}}',\bm{j}')] \notag \\
=& \sum_{\bm{k} \in [N]^n} 
\sum_{
\begin{subarray}{c} \bm{k}' \in [N]^n \\ \bm{k}' \sim \bm{k} \end{subarray}}
\sum_{\begin{subarray}{c} \sigma \in S_{2n} \\ \tilde{\bm{k}}' = 
(\tilde{\bm{k}})^{\sigma} \end{subarray}}
\sum_{\begin{subarray}{c} \tau \in S_{2n} \\ \bm{j}' = \bm{j}^{\tau} \end{subarray}}
\mr{Wg}^{U(N)}_{2n}(\sigma \tau^{-1}). \notag
\end{align}

\subsection{Diagonal entry}

\begin{thm} (see Theorem \ref{thm:COE-diagonal-moment}.)
Let $V=(v_{ij})_{1 \le i,j \le N}$ be an $N \times N$ COE matrix.
For positive integers $i,N,n$ with $1 \le i \le N$,
$$
\bE[|v_{ii}|^{2n}]= \frac{2^n n!}{(N+1)(N+3) \cdots (N+2n-1)}.
$$
\end{thm}

\begin{proof}
The equation \eqref{eq:COEexpand1} gives
$$
\bE[|v_{ii}|^{2n}]=
\sum_{\bm{k} \in [N]^n} 
\sum_{
\begin{subarray}{c} \bm{k}' \in [N]^n \\ \bm{k}' \sim \bm{k} \end{subarray}}
\bE[U(\tilde{\bm{k}}, (i^{2n})| \tilde{\bm{k}}',(i^{2n}))]
=\sum_{\bm{k} \in [N]^n} 
\bE[U(\tilde{\bm{k}}, (i^{2n})| \tilde{\bm{k}},(i^{2n}))]
\sum_{
\begin{subarray}{c} \bm{k}' \in [N]^n \\ \bm{k}' \sim \bm{k} \end{subarray}}
1.
$$
Given a sequence $\bm{k}$ of type $\mu$,
the number of sequences $\bm{k}'$ in $[N]^n$ with $\bm{k} \sim 
\bm{k}'$ is $\frac{n!}{\mu!}$, and
the type of  
$\tilde{\bm{k}}$
is $2\mu$.
Hence,
by Proposition \ref{CUE:row1},
 the last equation equals
\begin{align*}
& \frac{n!}{N(N+1)(N+2) \cdots (N+2n-1)}
 \sum_{\mu \vdash n} \sum_{\begin{subarray}{c} \bm{k} \in [N]^{n} \\ \text{(type of $\bm{k}$)$=\mu$} \end{subarray}}
\frac{(2\mu)!}{\mu!} \\
=&
\frac{2^n n!}{N(N+1)(N+2) \cdots (N+2n-1)}
 \sum_{\mu \vdash n} \sum_{\begin{subarray}{c} \bm{k} \in [N]^{n} \\ \text{(type of $\bm{k}$)$=\mu$} \end{subarray}}
\prod_{j=1}^{\ell(\mu)} (2\mu_j-1)!!. 
\end{align*}
To end the proof of the theorem,
it is enough to show the identity
\begin{equation} \label{eq:numerator}
\sum_{\mu \vdash n} 
\sum_{\begin{subarray}{c} \bm{k} \in [N]^{n} \\ \text{(type of $\bm{k}$)$=\mu$} \end{subarray}}
\prod_{j=1}^{\ell(\mu)} (2\mu_j-1)!!
=N(N+2)(N+4) \cdots (N+2n-2),
\end{equation}
which will be proved in the next subsection.
\end{proof}

\begin{cor}
Let $V^{(N)}=(v_{ij}^{(N)})$, $N \ge 1$, be a sequence of COE matrices.
Fix positive integers $i$ and $n$. As $N \to \infty$, 
$$
\bE[|v_{ii}^{(N)}|^{2n}] = 2^n n! (N^{-n} -n^2 N^{-n-1}) +O(N^{-n-2}).
$$
\end{cor}

\begin{cor}(see Theorem \ref{cor:limitGaussian}.) \label{cor:limitdist1}
Let $V^{(N)}=(v_{ij}^{(N)})$, $N \ge 1$, be a sequence of COE matrices. Fix 
a positive integer $i$.
As $N \to \infty$,
the random variable $\sqrt{N/2} v_{ii}^{(N)}$ converges to
a standard complex Gaussian random variable in distribution.
\end{cor}

\begin{proof}
Recall that a random variable $Z$ distributed to
a standard complex normal distribution satisfies 
$\bE [Z^n \overline{Z}^m]= \delta_{n,m} n!$ with any positive integers $n,m$.
Consider a random variable $Z_N=\sqrt{N/2} v_{ii}^{(N)}$.
From the last corollary and Lemma \ref{lem:COEcondition} we have 
$\bE ((Z_N)^m (\overline{Z_N})^n)=0$ ($m \not= n$) and 
$\lim_{N \to \infty} \bE ((Z_N)^n (\overline{Z_N})^n) = n!$.
This implies that random variables$\{Z_N\}_{N \ge 1}$ converge in distribution 
to  a standard complex normal variable (see, e.g., the proof of Theorem 3.1 in \cite{PR}).
\end{proof}

\subsection{Combinatorial lemmas}

The purpose in this subsection is to give a  remaining proof of  
 \eqref{eq:numerator}.

Given a non-empty finite set $I$ of positive integers,
denote by $\mcal{M}(I)$ the set of all unordered pairings on
$\bigsqcup_{i \in I} \{2i-1,2i\}$.
For example, $\mcal{M}(\{1,3\})$ consists of three pairings
$$
\{\{1,2\},\{5,6\}\}, \qquad \{\{1,5\},\{2,6\}\}, \qquad
\{\{1,6\},\{2,5\}\}. 
$$

Given $\mf{m} \in \mcal{M}([n])$, we attach a graph $\Gamma(\mf{m})$ with vertices
$1,2,\dots,2n $ and with the edge set
$$
\big\{ \{2k-1,2k\} \ | \ k \in [n] \big\} \sqcup  \mf{m}.
$$
Denote by $\kappa(\mf{m})$ the number of 
connected components of $\Gamma(\mf{m})$.

\begin{lem} \label{lem1}
For positive integers $n,N$, we have
$$
\sum_{\mf{m} \in \mcal{M}([n])} N^{\kappa(\mf{m})}= N(N+2)(N+4) \cdots (N+2n-2).
$$
\end{lem}

\begin{proof}
We show it by induction on $n$.
When $n=1$, we have $\sum_{\mf{m} \in \mcal{M}([1])} N^{\kappa(\mf{m})}=N$,
so that the claim holds true.

Suppose the statement at $n$ holds true and 
consider the $n+1$ case.
For $k=1,2,\dots,2n+1$, we set 
$$
\mcal{M}_k([n+1])= \{\mf{m} \in \mcal{M}([n+1]) \ | \ \{k,2n+2 \} \in \mf{m}\}.
$$
We construct a bijection $\mcal{M}_k([n+1]) \ni \mf{m} \mapsto \mf{m}' \in \mcal{M}([n])$ as follows.
If $k=2n+1$ and $\mf{m} \in \mcal{M}_k([n+1])$, we define $\mf{m}'$ to be 
the pairing obtained by removing $\{2n+1,2n+2\}$ from $\mf{m}$.
If $1 \le k \le 2n$ and $\mf{m} \in \mcal{M}_k([n+1])$, we define $\mf{m}'$ to be 
the pairing obtained by removing $\{k,2n+2\}$ and a pair $\{i,2n+1\}$ (with some $i$)
from $\mf{m}$ and by adding
$\{i,k\}$.
It is easy to see that the map $\mcal{M}_k([n+1]) \ni \mf{m} \mapsto \mf{m}' \in \mcal{M}([n])$
is bijective
 and that
$$
\kappa(\mf{m}) = \begin{cases}\kappa(\mf{m}')+1 & \text{if $\mf{m} \in \mcal{M}_{2n+1}([n+1])$},\\
\kappa(\mf{m}') & \text{if $\mf{m} \in \mcal{M}_{k}([n+1])$ and $k=1,2,\dots,2n$}.
\end{cases} 
$$
Hence it follows from the induction assumption that
\begin{align*}
\sum_{\mf{m} \in \mcal{M}([n+1])} N^{\kappa(\mf{m})} =& 
\sum_{\mf{m} \in \mcal{M}_{2n+1}([n+1])} N^{\kappa(\mf{m})}+
\sum_{k=1}^{2n} \sum_{\mf{m} \in \mcal{M}_k([n+1])} N^{\kappa(\mf{m})} \\
=&\sum_{\mf{n} \in \mcal{M}([n])} N^{\kappa(\mf{n})+1} +
\sum_{k=1}^{2n} \sum_{\mf{n} \in \mcal{M}([n])} N^{\kappa(\mf{n})}   \\
=& N(N+1) \cdots (N+2n-2) (N+2n),
\end{align*}
from which the $n+1$ case follows.
\end{proof}

\begin{lem} \label{lem2}
Let $\mu$ be a partition of $n$ and 
let $(i_1,\dots,i_{2n}) \in [N]^{2n}$ be a sequence of type 
$2\mu=(2\mu_1,2\mu_2,\dots)$.
Then 
$$
\sum_{\mf{m} \in \mcal{M}([n])} 
\(\prod_{\{p,q\} \in \mf{m}} \delta_{i_{p},i_{q}}\)
= \prod_{j=1}^{\ell(\mu)} (2\mu_j-1)!!.
$$
\end{lem}

\begin{proof}
Given a partition $\mu$ of $n$, we define
$\mcal{M}(\mu)$ to be the set of all pairings $\mf{m}$ in $\mcal{M}([n])$ of the form
$$
\mf{m}= \mf{m}^{(1)} \sqcup \mf{m}^{(2)} \sqcup \cdots,
$$ 
where
$\mf{m}^{(1)}  \in \mcal{M}(\{1,2,\dots,\mu_1\})$, 
$\mf{m}^{(2)}  \in \mcal{M}(\{\mu_1+1,\mu_1+2,\dots,\mu_1+\mu_2\})$,
and so on.

In order to prove the lemma,
we can suppose
$$
(i_1,\dots,i_{2n}) =\bm{i}_{2\mu}=(\underbrace{1,1,\dots,1}_{2\mu_1},
\underbrace{2,2,\dots,2}_{2\mu_2},\dots, \underbrace{l,l,\dots,l}_{2\mu_l})
$$
because of 
$$
\sum_{\mf{m} \in \mcal{M}([n])} 
\(\prod_{\{p,q\} \in \mf{m}} \delta_{i_{p},i_{q}}\)
=\frac{1}{2^n n!} \sum_{\sigma \in S_{2n}} \prod_{j=1}^{n} 
\delta_{i_{\sigma(2j-1)}, i_{\sigma(2j)}}.
$$
In that case, for each $\mf{m} \in \mcal{M}([n])$,
the product $\prod_{\{p,q\} \in \mf{m}} \delta_{i_{p},i_{q}}$
does not vanish if and only if $\mf{m}\in \mcal{M}(\mu)$.
It follows the result since $|\mcal{M}(\mu)|= \prod_{j=1}^{\ell(\mu)} (2\mu_j-1)!!$.
\end{proof}

\begin{proof}[Proof of Equation \eqref{eq:numerator}]
From Lemma \ref{lem2},
\begin{align*}
\sum_{\mu \vdash n} 
\sum_{\begin{subarray}{c} \bm{k} \in [N]^{n} \\ \text{(type of $\bm{k}$)$=\mu$} \end{subarray}}
\prod_{j=1}^{\ell(\mu)} (2\mu_j-1)!!
=& \sum_{\mu \vdash n} 
\sum_{\begin{subarray}{c} \bm{k} \in [N]^n \\ \text{(type of $\bm{k}$)$=\mu$} \end{subarray}} 
\sum_{\mf{m} \in \mcal{M}([n])}
\prod_{\{p,q\} \in \mf{m}}\delta_{\tilde{k}_{p}, \tilde{k}_{q}} \\
=& \sum_{\mf{m} \in \mcal{M}([n])} \sum_{\bm{k} \in [N]^n} 
\prod_{\{p,q\} \in \mf{m}}\delta_{\tilde{k}_{p}, \tilde{k}_{q}}.
\end{align*}
Since it is immediate to see 
$$
\sum_{\bm{k} \in [N]^n} 
\prod_{\{p,q\} \in \mf{m}}\delta_{\tilde{k}_{p}, \tilde{k}_{q}}
= N^{\kappa(\mf{m})},
$$
it follows from Lemma \ref{lem1} that
$$
\sum_{\mu \vdash n} 
\sum_{\begin{subarray}{c} \bm{k} \in [N]^{n} \\ \text{(type of $\bm{k}$)$=\mu$} \end{subarray}}
\prod_{j=1}^{\ell(\mu)} (2\mu_j-1)!!
= \sum_{\mf{m} \in \mcal{M}([n])} N^{\kappa(\mf{m})} = 
N(N+2) \cdots (N+2n-2),
$$
so that \eqref{eq:numerator} has been proved. It completes the proof of Theorem
\ref{thm:COE-diagonal-moment}.
\end{proof}

\subsection{Off-diagonal entry }

Let $1 \le i \not=j \le N$.
Put $\bm{i}=(i,j,i,j,i,j,\dots,i,j) \in [N]^{2n}$.
In a similar way to a diagonal-entry case,
we have
\begin{align*}
\bE[|v_{ij}|^{2n}]
=& 
\sum_{\mu \vdash n} \frac{n!}{\mu!}
\sum_{
\begin{subarray}{c} \bm{k} \in [N]^n \\ \text{(type of $\bm{k}$) is $\mu$}
\end{subarray}
}
\bE[ U(\tilde{\bm{k}}, \bm{i}|\tilde{\bm{k}},\bm{i})] \\
=& \sum_{\mu \vdash n} \frac{n!}{\mu!}
\sum_{
\begin{subarray}{c} \bm{k} \in [N]^n \\ \text{(type of $\bm{k}$) is $\mu$}
\end{subarray}
}
\sum_{\sigma \in S_{2\mu}} \sum_{\tau \in S^{*}_{2n}}
\mr{Wg}^{U(N)}_{2n}(\sigma \tau).
\end{align*}
where 
$$
S_{2n}^{*}=S_{\{1,3,5,\dots,2n-1\}} \times
S_{\{2,4,6,\dots,2n\}}
$$
is the stabilizer subgroup for $\bm{i}$.
Since the number of sequences in $[N]^n$ of type $\mu$
is
$$
\frac{n!}{\mu!} \frac{N!}{\prod_{k \ge 1} m_k(\mu)! \cdot (N-\ell(\mu))!},
$$
we have

\begin{prop} \label{prop:offdiagonal-sum}
Let $V=(v_{ij})_{1 \le i,j \le N}$ be a COE matrix.
For positive integers $i,j,N,n$ with $1 \le i \not=j \le N$,
\begin{equation} \label{eq:offdiagonal-sum}
\bE[|v_{ij}|^{2n}] =\sum_{\mu \vdash n} 
\frac{(n!)^2}{(\mu!)^2 \prod_{k \ge 1} m_k(\mu)!}
N(N-1) \cdots (N-\ell(\mu)+1) W(\mu,N)
\end{equation}
where
\begin{equation}
W(\mu,N)=\sum_{\sigma \in S_{2\mu}} \sum_{\tau \in S^{*}_{2n} }
\mr{Wg}^{U(N)}_{2n}(\sigma \tau).
\end{equation}
\end{prop}

To see the asymptotics for $\bE[|v_{ij}|^{2n}]$,
we focus on $W(\mu,N)$ in the limit $N \to \infty$.

\begin{lem} \label{lem:asymW}
Let $\mu \vdash n$. As $N \to \infty$,
$$
W(\mu,N)=
\begin{cases}
N^{-2n}-n^2 N^{-2n-1} +O(N^{-2n-2})
& \text{if $\mu=(1^n)$} \\
4N^{-2n} +O(N^{-2n-1}) & \text{if $\mu=(2,1^{n-2})$} \\
O(N^{-2n}) & \text{otherwise}.  
\end{cases}
$$
\end{lem}
\begin{proof}
We first consider the case where $\mu=(1^n)$.
Let $\sigma \in S_{(2^n)}=S_{\{1,2\}} \times S_{\{3,4\}} \times \cdots 
\times S_{\{2n-1,2n\}}$ and $\tau \in S_{2n}^{*}$.
It is easy to see that
if $\sigma \tau=\mr{id}_{2n}$ then $\sigma=\tau=\mr{id}_{2n}$, and that
if $\sigma \tau$ is a transposition then either
$\sigma=\mr{id}_{2n}$ or $\tau=\mr{id}_{2n}$.
Hence, from Lemma \ref{lem:asymWg} it follows that
$W((1^n),N)=N^{-2n}
- c N^{-2n-1}
+O(N^{-2n-2})$,
where 
$$
c=|\{\sigma \in S_{(2^n)} \ | \ \text{transpositions}\}| +
|\{\tau \in S_{2n}^{*} \ | \ \text{transpositions}\}|
=n+2\binom{n}{2}=n^2,
$$
so that the result at $\mu=(1^n)$ follows.

We next consider the case where $\mu=(2,1^{n-2})$.
Let $\sigma \in S_{2\mu}=
S_{\{1,2,3,4\}} \times S_{\{5,6\}} \times \cdots 
\times S_{\{2n-1,2n\}}$ and $\tau \in S_{2n}^{*}$.
Since $S_{2\mu} \cap S_{2n}^{*}=
\{ \mr{id}_{2n}, (1 \ 3),  (2 \ 4), (1 \ 3) (2 \ 4)\}$, 
the number of pairs $(\sigma,\tau)$ satisfying $\sigma \tau=\mr{id}_{2n}$
is equal to $4$.
Hence, from Lemma \ref{lem:asymWg}, we have
$W((2,1^{n-2}),N)= 4N^{-2n}
+O(N^{-2n-1})$.

The case where $\mu \not=(1^n), (2,1^{n-2})$ is clear from Lemma \ref{lem:asymWg}.
\end{proof}

We have the following leading and subleading terms in $\bE[|v_{ij}|^{2n}]$
in the limit $N \to \infty$.

\begin{thm}(See Theorem \ref{thm:COE-offdiagonal-moment}.)
Fix positive integers $i,j,n$ with $i \not= j$. 
Let $V^{(N)}=(v_{ij}^{(N)})$, $N \ge 1$,  be a sequence of COE matrices.
As $N \to \infty$,
$$
\bE[|v_{ij}^{(N)}|^{2n}]= n! N^{-n} - n! \frac{n(n+1)}{2}  N^{-n-1}
+O(N^{-n-2}). 
$$
\end{thm}

\begin{proof}
We divide the sum in \eqref{eq:offdiagonal-sum}
into $\mu=(1^n)$, $\mu=(2,1^{n-2})$, and  other $\mu$'s.
Since
$$
N(N-1) \cdots (N-\ell(\mu)+1)= 
\begin{cases}
N^n - \binom{n}{2} N^{n-1}+ O(N^{n-2})
& \text{if $\mu=(1^n)$}, \\
N^{n-1} +O(N^{n-2}) & \text{if $\mu=(2,1^{n-2})$}, \\
O(N^{n-2}) & \text{otherwise},  
\end{cases}
$$
it follows from Lemma \ref{lem:asymW} that
\begin{align*}
\bE[|v_{ij}^{(N)}|^{2n}] =& n! \(N^n - \binom{n}{2} N^{n-1} +O(N^{n-2})\)
(N^{-2n} -n^2 N^{-2n-1} +O(N^{-2n-2})) \\
&+ \frac{n!}{4} n(n-1)  (N^{n-1} +O(N^{n-2})) (4N^{-2n} +O(N^{-2n-1})) \\
&+ \sum_{\begin{subarray}{c} \mu \vdash n \\ \mu \not= (1^n), (2,1^{n-2}) 
\end{subarray}} O(N^{n-2}) \cdot O(N^{-2n}),
\end{align*}
from which a straightforward calculation gives the result.
\end{proof}

\begin{cor}(see Theorem \ref{cor:limitGaussian}.)
Fix $i \not= j$.
As $N \to \infty$,
a sequence of the random variables $\{ \sqrt{N} v_{ij}^{(N)} \}_{N \ge 1}$  converges to
a standard complex Gaussian random variable in distribution.
\end{cor}

\begin{proof}
The proof is similar to that of Corollary \ref{cor:limitdist1}.
\end{proof}

\appendix

\section{Appendix: Examples of lower degrees}

Let $V=(v_{ij})_{1 \le i,j \le N}$ be a COE matrix.
We give examples of joint moments of small degrees.

\subsection{Degree 1}

\begin{example}
Let $j_1,j_2,j_1',j_2' \in [N]$. 
Consider the average $\bE[v_{j_1 j_2} \overline{v_{j_1'j_2'}}]$. 
\begin{enumerate}
\item $\bE[|v_{ii}|^2] = \frac{2}{N+1}$ for each $1 \le i \le N$;  
\item $\bE[|v_{ij}|^2] = \frac{1}{N+1}$ for $ 1 \le i \not= j \le N$;
\item $\bE[v_{j_1 j_2} \overline{v_{j_1' j_2'}}]=0$
unless $(j_1,j_2) \sim (j_1',j_2')$. 
\end{enumerate}
\end{example}
 
The first case follows from Theorem \ref{thm:COE-diagonal-moment}.
The second case follows from Proposition \ref{prop:offdiagonal-sum}:
$$
\bE[|v_{ij}|^2]=N \cdot W((1),N)= N \sum_{\sigma \in S_2} \mr{Wg}^{U(N)}_2(\sigma) =\frac{N}{N(N+1)}
= \frac{1}{N+1}.
$$
The rest follows from Lemma \ref{lem:COEcondition}.

\subsection{Degree 2}

Let $ 1 \le i \not=j \le N$ and let $j_1,j_2,j_3,j_4 \in [N]$ be distinct.
\begin{align}
\bE[|v_{ii}|^4]=& \bE[|v_{11}|^4]=\frac{8}{(N+1)(N+3)}.  \label{degree2-1}\\
\bE[ |v_{ij}|^{4}] =&\bE[ |v_{12}|^{4}]= \frac{2}{N(N+3)}. \label{degree2-2}\\
\bE[|v_{ii} v_{jj}|^2]=&\bE[|v_{11} v_{22}|^2]= \frac{4(N+2)}{N(N+1)(N+3)}. \label{degree2-3}\\
\bE[v_{ij}^2 \overline{v_{ii} v_{jj}}]=& 
\bE[v_{12}^2 \overline{v_{11} v_{22}}]=\frac{-4}{N(N+1)(N+3)}. \label{degree2-4} \\
\bE[|v_{j_1 j_2} v_{j_3 j_4}|^2] =& \bE[|v_{12} v_{34}|^2] = \frac{N+2}{N(N+1)(N+3)}.
\label{degree2-5}
\end{align}

The equation \eqref{degree2-1} follows from Theorem \ref{thm:COE-diagonal-moment}.

For each partition $\mu$ of $n$, put
$\mr{Wg}_n^{U(N)}(\mu):=\mr{Wg}_n^{U(N)}(\sigma)$,
where $\sigma$ is a permutation in $S_{n}$ of cycle-type $\mu$.
We use the following explict expressions:
\begin{align*}
\mr{Wg}_4^{U(N)}((4))=& \frac{-5}{N(N^2-1)(N^2-4)(N^2-9)}; \\
\mr{Wg}_4^{U(N)}((3,1))=& \frac{2N^2-3}{N^2(N^2-1)(N^2-4)(N^2-9)}; \\
\mr{Wg}_4^{U(N)}((2,2))=& \frac{N^2+6}{N^2(N^2-1)(N^2-4)(N^2-9)}; \\
\mr{Wg}_4^{U(N)}((2,1,1))=& \frac{-1}{N(N^2-1)(N^2-9)}; \\
\mr{Wg}_4^{U(N)}((1^4))=& \frac{N^4-8N^2+6}{N^2(N^2-1)(N^2-4)(N^2-9)}.
\end{align*}

\begin{proof}[Proof of \eqref{degree2-2}]
We apply Proposition \ref{prop:offdiagonal-sum}.
Since 
\begin{align*}
W((1^2),N)=& \sum_{\sigma \in <(1 \ 2), (3 \ 4)>} 
\sum_{\tau \in <(1 \ 3), (2 \ 4)>} \mr{Wg}^{U(N)}_4(\sigma \tau)  \\
=& \mr{Wg}^{U(N)}_4 ((1^4)) + 4 \mr{Wg}^{U(N)}_4 ((2,1,1)) + 3\mr{Wg}^{U(N)}_4 ((2,2)) \\
& \quad  + 4 \mr{Wg}^{U(N)}_4 ((3,1)) +4\mr{Wg}^{U(N)}_4 ((4)) \\
=& \frac{N^2 +N+2}{N^2(N-1)(N+1)(N+2)(N+3)}
\end{align*}
and 
$$
W((2),N)= \sum_{\sigma \in S_4} \sum_{\tau \in S_4^{*}} \mr{Wg}^{U(N)}_4(\sigma \tau)
= \frac{4}{N(N+1)(N+2)(N+3)},
$$
 Proposition \ref{prop:offdiagonal-sum} gives
$$
\bE[ |v_{ij}|^{4}] = 2 N (N-1) \cdot W((1^2),N)+ N \cdot  W((2),N) 
= \frac{2}{N(N+3)}.
$$
\end{proof}

Consider the average $\bE[v_{j_1 j_2}v_{j_3 j_4} \overline{v_{j_1'j_2'} v_{j_3' j_4'}}]$,
where $\bm{j}=(j_1,j_2,j_3,j_4) \sim \bm{j}'=(j_1',j_2',j_3',j_4')$.
The equation \eqref{eq:COEexpand1} and a direct calculation give
\begin{align}
& \bE[v_{j_1 j_2}v_{j_3 j_4} \overline{v_{j_1'j_2'} v_{j_3' j_4'}}]  
\label{eq:moment2}
\\
=&\sum_{\nu \vdash 2} 
\sum_{\begin{subarray}{c}\bm{k} \in [N]^2 \\ \text{type $\nu$} \end{subarray}} 
\sum_{
\begin{subarray}{c} \bm{k}' \in [N]^2 \\ \bm{k}' \sim \bm{k} \end{subarray}}
\bE[U(\tilde{\bm{k}},\bm{j}| \tilde{\bm{k}}', \bm{j}')] \notag \\
=&  
 N(N-1) \{\bE[U((1,1,2,2),\bm{j}|(1,1,2,2),\bm{j}')] 
 +\bE[U((1,1,2,2),\bm{j}|(2,2,1,1),\bm{j}')]\} \notag \\
 &+N \cdot \bE [U((1^4),\bm{j}|(1^4),\bm{j})]. \notag
\end{align}
We note that 
$\bE [U((1^4),\bm{j}|(1^4),\bm{j})]= \frac{\mu!}{N(N+1)(N+2)(N+3)}$
if $\bm{j}$ is of type $\mu \vdash 4$.

\begin{proof}[Proof of \eqref{degree2-3}]
From Lemma \ref{lem:WgFormula}, we have
\begin{align*}
& \bE[U((1,1,2,2),(1,1,2,2)|(1,1,2,2),(1,1,2,2)] \\
=& 4 \sum_{\sigma \in S_{(2,2)}} \mr{Wg}^{U(N)}_4(\sigma) 
= 4 \{ \mr{Wg}^{U(N)}_4 ((1^4))+ 2\mr{Wg}^{U(N)}_4 ((2,1,1)) +\mr{Wg}^{U(N)}_4 ((2,2)) \} \\
=&
\frac{4}{N^2(N-1)(N+3)}
\end{align*}
and
\begin{align*}
&\bE[U((1,1,2,2),(1,1,2,2)|(2,2,1,1),(1,1,2,2))] \\
=& \sum_{\sigma ,\tau \in S_{(2,2)}} \mr{Wg}^{U(N)}_4 (\sigma (1\ 3)(2\ 4) \tau^{-1}) 
= 4 \sum_{\sigma \in S_{(2,2)}} \mr{Wg}^{U(N)}_4 ( (1\ 3)(2\ 4) \sigma) \\
=& 4 \{2 \mr{Wg}^{U(N)}_4((2,2))+2\mr{Wg}^{U(N)}_4((4))\} 
= \frac{8}{N^2 (N^2-1)(N+2)(N+3)}.
\end{align*}
Hence the equation \eqref{eq:moment2} gives
\begin{align*}
\bE[|v_{ii} v_{jj}|^2]
=& N(N-1) \cdot \frac{4}{N^2 (N-1)(N+3)} + N(N-1) \cdot \frac{8}{N^2(N^2-1)(N+2)(N+3)} \\
&+ N \frac{4}{N(N+1)(N+2)(N+3)} \\
=& \frac{4(N+2)}{N(N+1)(N+3)}.
\end{align*}
\end{proof}

\begin{proof}[Proof of \eqref{degree2-4}]
From Lemma \ref{lem:WgFormula}, we have
\begin{align*}
&\bE[U((1,1,2,2),(1,2,1,2)|(1,1,2,2),(1,1,2,2))] \\
=& \sum_{\sigma,\tau \in S_{(2,2)}} \mr{Wg}^{U(N)}_4(\sigma ((2\ 3)\tau )^{-1})  \\
=& 4 \sum_{\sigma \in S_{(2,2)}} \mr{Wg}^{U(N)}_4 (\sigma (2 \ 3))  \\
=& 4 \left\{ \mr{Wg}^{U(N)}_4((2,1,1))+2 \mr{Wg}^{U(N)}_4 ((3,1))+\mr{Wg}^{U(N)}_4 ((4)) \right\} \\
=& \frac{-4}{N^2 (N-1)(N+2)(N+3)}.
\end{align*}
Similarly, we have 
$$
\bE[U((1,1,2,2),(1,2,1,2)|(2,2,1,1),(1,1,2,2))] =\frac{-4}{N^2 (N-1)(N+2)(N+3)}.
$$
Hence the equation \eqref{eq:moment2} gives
\begin{align*}
\bE[v_{ij}^2 \overline{v_{ii} v_{jj}}]
=& 2N(N-1) \cdot \frac{-4}{N^2 (N-1)(N+2)(N+3)} + N \cdot \frac{4}{N(N+1)(N+2)(N+3)} \\
=& \frac{-4}{N(N+1)(N+3)}.
\end{align*}
\end{proof}

\begin{proof}[Proof of \eqref{degree2-5}]
In a similar way to the proof of \eqref{degree2-3}, we have
\begin{align*}
&\bE[U((1,1,2,2),(1,2,3,4)|(1,1,2,2),(1,2,3,4))] \\
=& \sum_{\sigma \in S_{(2,2)}} \mr{Wg}^{U(N)}_4 (\sigma) 
= \frac{1}{N^2 (N-1)(N+3)}
\end{align*}
and
\begin{align*}
&\bE[U((1,1,2,2),(1,2,3,4)|(2,2,1,1),(1,2,3,4))] \\
=& \sum_{\sigma \in S_{(2,2)}} \mr{Wg}^{U(N)}_4 ( \sigma(1\ 3)(2\ 4) ) 
= \frac{2}{N^2 (N^2-1)(N+2)(N+3)}.
\end{align*}
Hence the equation \eqref{eq:moment2} gives
\begin{align*}
\bE[|v_{j_1 j_2} v_{j_3 j_4}|^2] =& \bE[|v_{12} v_{34}|^2]  \\
=& N(N-1) \cdot \frac{1}{N^2(N-1)(N+3)} + N(N-1) \cdot \frac{2}{N^2 (N^2-1)(N+2)(N+3)} \\
&+ N \cdot \frac{1}{N(N+1)(N+2)(N+3)} \\ 
=& \frac{N+2}{N(N+1)(N+3)}. 
\end{align*}\end{proof}

\subsection{Moments of traces}

\begin{align*}
\bE[ |\tr (V)|^4] =& \frac{8(N^2+2N-2)}{(N+1)(N+3)}, \\
\bE[|\tr(V^2)|^2] =& \frac{4(N^2+2N-1)}{(N+1)(N+3)}, \\
\bE[\tr(V^2) \overline{\tr(V)^2}]=& \frac{8}{(N+1)(N+3)},
\end{align*}
which are obtained in \cite[Appendix]{JM} from symmetric function theory.
If we use results in the previous subsection,
another proof is given as follows.
\begin{align*}
\bE[ |\tr (V)|^4] =& \sum_{i_1, i_2, i_3, i_4} \bE[ v_{i_1 i_1} v_{i_2 i_2}
\overline{ v_{i_3 i_3} v_{i_4 i_4}}] \\
=& N \bE[|v_{11}|^4] + 2 N(N-1) \bE[|v_{11}v_{22}|^2] 
= \frac{8(N^2+2n-2)}{(N+1)(N+3)};
\end{align*}
\begin{align*}
\bE[ |\tr (V^2)|^2] =& \sum_{i_1, i_2, i_3, i_4} \bE[ v_{i_1 i_2}^2 \overline{v_{i_3 i_4}^2}] \\
=& N \bE[|v_{11}|^4] + 2 N(N-1) \bE[ |v_{12}|^4]
= \frac{4(N^2+2N-1)}{(N+1)(N+3)};
\end{align*}
\begin{align*}
\bE[\tr(V^2) \overline{\tr(V)^2}]=&
\sum_{i_1, i_2, i_3, i_4} \bE [ v_{i_1 i_2}^2 \overline{v_{i_3 i_3} v_{i_4 i_4}}] \\
=& N \bE [|v_{11}|^4] + 2N(N-1) \bE[ v_{12}^2 \overline{v_{11} v_{22}}]
= \frac{8}{(N+1)(N+3)}.
\end{align*}


\end{document}